\documentclass[11pt,reqno]{amsart}
\usepackage{graphicx}
\usepackage{enumitem}
\usepackage{comment}
\usepackage{caption}
\usepackage{setspace}
\setlist[enumerate]{font={\upshape}, label=\arabic*., leftmargin=2.5em}
\setlist[itemize]{leftmargin=2.5em}
\setlist[description]{leftmargin=\parindent, 
	itemsep=3pt
}

\newlist{equivlist}{enumerate}{1}
\setlist[equivlist]{font={\upshape}, label=(\roman*)}
\usepackage{tikz}
\usetikzlibrary{matrix,intersections,calc}
\tikzset{ 
	table/.style={
		matrix of nodes,
		nodes={rectangle,text width=1.75em,align=center},
		text depth=1.25ex,
		text height=2.5ex,
		nodes in empty cells
	}
}
\usepackage{mathtools}
\usepackage{amsfonts}
\usepackage{amssymb}
\usepackage{cite} 
\usepackage[sort]{natbib}

\usepackage[pdftex
,pagebackref
]{hyperref}
\hypersetup{
	colorlinks=true,
	allcolors=blue
}
\usepackage[capitalize]{cleveref}
\usepackage[a4paper,margin=0.8in]{geometry}
\newtheorem{theorem}{Theorem}[section]
\newtheorem{lemma}[theorem]{Lemma}
\newtheorem{claim}{Claim}[theorem]
\Crefname{claim}{Claim}{Claims}
\newlist{lemlist}{enumerate}{1}
\setlist[lemlist]{font={\upshape}, label={\upshape(\alph*)},ref={\thelemma(\alph*)},leftmargin=*}

\newtheorem{conjecture}[theorem]{Conjecture}
\Crefname{conjecture}{Conjecture}{Conjectures}

\expandafter\let\expandafter\oldproof\csname\string\proof\endcsname
\let\oldendproof\endproof
\renewenvironment{proof}[1][\proofname]{%
	\oldproof[\normalfont\bfseries #1]%
}{\oldendproof}
\newenvironment{subproof}[1][\normalfont\it Subproof]{%
	\begin{proof}[#1]%
	}{%
	\end{proof}%
}

\newcommand{\dd}{\textquotedblleft}
\newcommand{\ee}{\textquotedblright}

\newcommand{\mac}{\mathcal}
\newcommand{\mab}{\mathbb}

\newcommand{\eps}{\varepsilon}

\renewcommand{\subset}{\subseteq}

\newcommand{\erh}{Erd\H{o}s--Hajnal}
\newcommand{\gas}{Gy\'{a}rf\'{a}s--Sumner}
\newcommand{\chis}{\chi^*}
\DeclarePairedDelimiter\abs{\lvert}{\rvert}%
\DeclarePairedDelimiter\ceil{\lceil}{\rceil}%
\makeatletter
\newcommand{\leqnomode}{\tagsleft@true}
\newcommand{\reqnomode}{\tagsleft@false}
\makeatother

\begin{document}
	\title{Fractionally colouring $P_5$-free graphs}
	\author{Tung H. Nguyen}
	\address{Mathematical Institute and Christ Church, University of Oxford, Oxford, UK}
	\email{\href{mailto:nguyent@maths.ox.ac.uk}{nguyent@maths.ox.ac.uk}
	}
	\thanks{Part of this work was conducted while the author was at Princeton University and was partially supported by AFOSR grant FA9550-22-1-0234, NSF grant DMS-2154169,
		and a Porter Ogden Jacobus Fellowship.}
	\begin{abstract}
		We obtain some $d\ge2$ such that every graph $G$ with no induced copy of the five-vertex path $P_5$ has at most $\alpha(G)\omega(G)^d$ vertices.
		This \dd off-diagonal Ramsey\ee{} statement implies that every such graph $G$ has fractional chromatic number at most $\omega(G)^d$, and is another~step towards the polynomial \gas{} conjecture for $P_5$.
		The proof uses the recent \erh{} result for $P_5$ and adapts a decomposition argument for $P_5$-free graphs developed by the author in an earlier paper.
		
	\end{abstract}
	\maketitle
	\section{Introduction}
	\subsection{Main result}
	All graphs in this paper are finite and with no loops or parallel edges.
	For a graph $G$ with vertex set $V(G)$ and edge set $E(G)$, let $\abs G:=\abs{V(G)}$.
	The {\em clique number} of $G$, denoted by $\omega(G)$, is the maximum size of a clique in $G$;
	and the {\em stability number} of $G$, denoted by $\alpha(G)$, is the maximum size of a stable set in $G$. The {\em chromatic number} of a graph $G$, denoted by $\chi(G)$, is the least $\ell\ge0$ such that there is a partition of $V(G)$ into $\ell$ stable sets in $G$;
	equivalently the chromatic number of $G$ is the least $\ell\ge0$ such that the vertices of $G$ can be coloured by $\ell$ colours such that no two adjacent vertices get the same colour.
	For $S\subset V(G)$, let $G[S]$ be the graph with vertex set $S$ and edge set $\{uv:u,v\in S,uv\in E(G)\}$.
	A graph $H$ is an {\em induced subgraph} of $G$ if there exists $S\subset V(G)$ with $H=G[S]$;
	and $G$ is {\em $H$-free} if it has no induced subgraph isomorphic to $H$.
	The \gas{} conjecture~\cite{MR382051,MR634555} (see~\cite{scott2022,MR4174126} for surveys) asserts that:
	\begin{conjecture}
		[\gas]
		\label{conj:gs}
		For every forest $T$, there is a function $f\colon \mab N\to\mab N$ such that $\chi(G)\le f(\omega(G))$ for all $T$-free graphs $G$.
	\end{conjecture}
	While most confirmed cases of this conjecture yield (super-)exponential functions $f$, the conjecture could even be true with polynomial $f$; formally:
	\begin{conjecture}
		[Polynomial \gas]
		\label{conj:pgs}
		For every forest $T$, there exists $d\ge2$ such that $\chi(G)\le\omega(G)^d$ for all $T$-free graphs $G$.
	\end{conjecture}
	
	(The case when $T$ is a path was asked independently by Esperet~\cite{esperet} and by Trotignon and Pham~\cite{MR3789676}; see also~\cite{2024p5,MR4472775} for state-of-the-art partial results.)
	This conjecture is related to the following open problem of Erd\H os and Hajnal~\cite{MR599767,MR1031262} (see~\cite{MR1425208,MR3150572} for surveys and~\cite{density3,density4,density6,density7,MR4563865} for several recently confirmed special cases):
	\begin{conjecture}
		[\erh]
		\label{conj:eh}
		For every graph $H$, there exists $c>0$ such that every $H$-free graph $G$ satisfies $\max(\alpha(G),\omega(G))\ge\abs G^c$.
	\end{conjecture}
	To explain, observe that if Conjecture~\ref{conj:pgs} holds for a forest $T$ then so does Conjecture~\ref{conj:eh}; indeed, for every $T$-free graph $G$:
	\begin{equation*}
		\omega(G)^d\ge\chi(G)\ge \frac{\abs G}{\alpha(G)}
		\quad\Rightarrow\quad
		\alpha(G)\omega(G)^d\ge\abs G
		\quad\Rightarrow\quad
		\max(\alpha(G),\omega(G))\ge\abs G^{\frac1{d+1}}.
	\end{equation*}
	
	Currently the five-vertex path $P_5$ is the smallest open case of Conjecture~\ref{conj:pgs}. Even though it was recently shown in~\cite{density7} that $P_5$ does indeed satisfy Conjecture~\ref{conj:eh}, the proof method there appears insufficient to imply that it also satisfies the middle inequality above for some universal $d\ge2$ (we will elaborate more on this in \cref{sec:sketch}).
	The main result of this paper verifies such a substantial \dd off-diagonal Ramsey\ee{} improvement:
	\begin{theorem}
		\label{thm:p5main}
		There exists $d\ge2$ such that every $P_5$-free graph $G$ satisfies $\alpha(G)\omega(G)^d\ge\abs G$.
	\end{theorem}
	
	Though we make no attempt to optimise $d$, our method seems unlikely to result in some $d\le10$.
	Still, \cref{thm:p5main} could be true with $d=2$, which would be optimal as one can take $G$ to be the complement of an $n$-vertex triangle-free graph with stability number at most $O(\sqrt{n\log n})$~\cite{MR1369063}.
	\subsection{Background and discussion}
	A notable feature of Conjecture~\ref{conj:gs} is that it no longer holds if one allows the forbidden graph $T$ to contain a cycle,
	because a classical probabilistic argument of Erd\H os~\cite{MR102081} yields high-chromatic graphs with no short cycle (in particular triangle-free).
	In fact, Erd\H os' construction gives much more (see also~\cite{MR3524748}):
	\begin{theorem}
		[Erd\H os]
		\label{thm:erdos}
		For every $g\ge 3$ and every sufficiently large $n$, there are $n$-vertex graphs with no cycle of length at most $g$ and with stability number at most $n^{\frac{g-1}{g}+o(1)}$.
	\end{theorem}
	(Spencer~\cite{MR491337} improved this slightly to $n^{\frac{g-2}{g-1}+o(1)}$.)
	There are explicit constructions of this type, for example the Ramanujan graphs by Lubotzky, Phillips, and Sarnak~\cite{MR963118} (with exponent $\frac{3g-1}{3g}$).
	
	For every two integers $k,w\ge1$ and every graph $H$, let $R_H(k,w)$ be the least integer $n\ge1$ such that every $n$-vertex $H$-free graph has a stable set of size at least $k$ or a clique of size at least $w$.
	Then Conjecture~\ref{conj:gs} would imply that every forest $T$ satisfies $R_T(k,w)\le Ck$ for some $C=C(T,w)>0$ (see~\cite{stable1} for more discussion on this open problem) and Conjecture~\ref{conj:eh} says that every graph $H$ satisfies $R_H(k,k)\le k^d$ for some $d=d(H)\ge1$.
	In view of \cref{thm:erdos}, in order to prove an \dd if and only if\ee{} statement for excluding an induced forest $T$ in the spirit of Conjecture~\ref{conj:gs},
	it suffices to prove a weakening of the conjecture that $R_T(k,w)\le k^{1+o(1)}$ for all forests $T$, where the $o(1)$ term contains the constant factors depending on $T$ and $w$.
	This relaxed statement was recently confirmed via the following result~\cite{stable1} (in this paper $\log$ denotes the binary logarithm):
	\begin{theorem}
		[Nguyen--Scott--Seymour]
		\label{thm:trees}
		For every $\eps>0$ and every forest $T$, there exists $\delta>0$ such that for every $T$-free graph $G$, either $\omega(G)\ge \delta\log\log\abs G$ or $\alpha(G)\ge\abs G^{1-\eps}$.
	\end{theorem}
	Stating that $R_{P_5}(k,w)\le kw^d$ for some universal $d\ge2$, \cref{thm:p5main} is thus a considerable strengthening of \cref{thm:trees} in the special case $T=P_5$.
	
	\cref{thm:p5main} can also be reformulated in terms of bounding the fractional chromatic number and the Hall ratio of $P_5$-free graphs by polynomial functions of their clique numbers.
	Recall that the {\em fractional chromatic number} of a graph $G$, denoted by $\chis(G)$, is the minimal $q\ge0$ such that there are nonnegative numbers $\{x_S:S\in\mac I(G)\}$ satisfying $\sum_{S\in \mac I(G)}x_S=q$ and $\sum_{v\ni S}x_S\ge1$ for all $v\in V(G)$; here $\mac I(G)$ denotes the collection of stable sets of $G$.
	Let $\psi(G)$ be $\abs G/\alpha(G)$ if $\abs G\ge1$ and be $0$ if $\abs G=0$;
	and the {\em Hall ratio} of $G$, denoted by $\rho(G)$, is defined as $\max_F\psi(F)$ where the maximum ranges through all induced subgraphs $F$ of $G$.
	Then $\omega(G)\le\rho(G)\le\chis(G)\le\chi(G)$.
	It is not hard to see that \cref{thm:p5main} is equivalent to $\rho(G)\le\omega(G)^d$ for all $P_5$-free graphs $G$;
	and with a little more work we can show that it is also equivalent to the following \dd polynomial $\chis$-boundedness\ee{} result for such graphs $G$:
	\begin{theorem}
		\label{thm:p5frac}
		There exists $d\ge2$ such that $\chis(G)\le \omega(G)^d$ for all $P_5$-free graphs $G$.
	\end{theorem}
	\begin{proof}
		[Proof of the equivalence of \cref{thm:p5main,thm:p5frac}]
		It suffices to show that \cref{thm:p5main} implies \cref{thm:p5frac}.
		To see this, let $d\ge2$ be given by \cref{thm:p5main}, and let $G$ be a non-null $P_5$-free graph.
		By linear programming duality (see~\cite[Theorem 7.3.2]{MR1829620} for instance), there exists $f\colon V(G)\to\mab N_0$ and a stable set $S$ of $G$ such that $f(V(G))>0$ and $f(V(G))\ge\chis(G)\cdot f(I)$ for all nonempty stable sets $I$ of $G$ (here $f(A)=\sum_{v\in A}f(v)$ for all $A\subset V(G)$) where equality holds for $I=S$.
		In particular $f(S)\ge f(I)$ for all stable sets $I$ of $G$.
		Let $J$ be the graph obtained from $G$ by blowing up each $v\in V(G)$ by a stable set of size $f(v)$.
		Then $\abs J=f(V(G))$, $\alpha(J)=f(S)$, $\omega(J)\le\omega(G)$, and $J$ is $P_5$-free since $P_5$ has no two vertices with a common neighbour.
		Then \cref{thm:p5main} implies $\psi(J)\le\omega(G)^d$;
		and so $\chis(G)=f(V(G))/f(S)=\psi(J)\le\omega(J)^d\le\omega(G)^d$ by the choice of $f$.
		This proves \cref{thm:p5frac}.
	\end{proof}
	
	In view of the above discussion, let us say that a graph $T$ is:
	\begin{itemize}
		\item {\em poly-$\chi$-bounding} if there exists $d=d(T)\ge1$ such that $\chi(G)\le\omega(G)^d$ for all $T$-free graphs $G$;
		
		\item {\em poly-$\chis$-bounding} if there exists $d=d(T)\ge1$ such that $\chis(G)\le\omega(G)^d$ for all $T$-free graphs $G$;
		
		\item {\em poly-$\rho$-bounding} if there exists $d=d(T)\ge1$ such that $\rho(G)\le\omega(G)^d$ for all $T$-free graphs $G$;
		
		\item {\em off-diagonal \erh{}} if for every $\eps>0$, there exists $\delta=\delta(T,\eps)>0$ such that every $T$-free graph $G$ satisfies either $\omega(G)\ge\abs G^\delta$ or $\alpha(G)\ge\abs G^{1-\eps}$; and
		
		\item {\em \erh{}} if there exists $c=c(T)>0$ such that $\max(\alpha(G),\omega(G))\ge\abs G^c$ for all $T$-free graphs~$G$.
	\end{itemize}
	
	Hence, $T$ is poly-$\chi$-bounding if and only if it satisfies Conjecture~\ref{conj:pgs}, is poly-$\rho$-bounding if and only if there exists $d=d(T)\ge1$ satisfying $R_T(k,w)\le kw^d$ for all $k,w\ge1$, and  is off-diagonal \erh{} if and only if for every $\eps>0$, there exists $d=d(T,\eps)\ge1$ such that $R_T(k,w)\le k^{1+\eps}w^d$ for all $k,w\ge1$.
	Therefore, the above five properties are listed in the order of decreasing strength.
	Also, one can repeat the above proof of the equivalence of \cref{thm:main,thm:p5frac} to show that if $T$ has no two vertices with a common neighbour then $T$ is poly-$\chis$-bounding if and only if it is poly-$\rho$-bounding.
	In addition, by \cref{thm:erdos}, if $T$ is off-diagonal \erh{} then it is a forest; and given a forest $T$ with the \erh{} property, showing that it is off-diagonal \erh{} would probably be a good start towards proving its poly-$\chi$-bounding property, and would also significantly improve \cref{thm:trees} for $T$.
	
	To show that $T$ is poly-$\chi$-bounding, another possible approach is to prove that $\chi(G)\le \chis(G)^d$ (or even $\chi(G)\le\rho(G)^d$) for all $T$-free graphs $G$ for some universal $d\ge1$.
	This would reduce verifying the poly-$\chi$-bounding property of $T$ to proving that it is poly-$\chis$-bounding (or poly-$\rho$-bounding, respectively);
	and in particular achieving this for $T=P_5$ would eventually show that Conjecture~\ref{conj:pgs} holds for $P_5$ via \cref{thm:p5frac}.
	Also, while it is unknown whether poly-$\chi$-bounding forests are closed under disjoint union,
	in \cref{sec:union} we will show that this holds for those satisfying the second, third, and fourth properties above
	(for the \erh{} property this is already known by a theorem of Alon, Pach, and Solymosi~\cite{MR1832443}).
	
	Let us make a couple of remarks on the five-vertex path $P_5$ in the context of the above five properties.
	First, as mentioned, it is currently open whether $P_5$ is poly-$\chi$-bounding; the best known bound in this direction
	$\chi(G)\le\omega(G)^{d\log\omega(G)/\log\log\omega(G)}$ was recently obtained in~\cite{2024p5}.
	Second, because we will employ the recently verified \erh{} property of $P_5$~\cite{density7} to deduce its poly-$\chis$-bounding property (\cref{thm:p5frac}),
	perhaps the poly-$\chi$-bounding property of $P_5$ (if this is true) could be proved in a similar manner.
	Another evidence for this is that in~\cite{2024p5}, the \erh{} property of $P_5$ was used to show that every $P_5$-free graph contains a \dd polynomial versus linear\ee{} high-$\chi$ complete pair of vertex subsets, which leads to the aforementioned best known $\chi$-bounding function.

	\section{Notation and some proof ideas}
	\label{sec:sketch}
	For a graph $G$, let $\overline G$ denote the complement graph of $G$.
	When there is no danger of ambiguity, for $S\subset V(G)$ we write $\alpha(S)$, $\omega(S)$, $\psi(S)$, $\rho(S)$ for $\alpha(G[S])$, $\omega(G[S])$, $\psi(G[S])$, $\rho(G[S])$ respectively.
	For disjoint $A,B\subset V(G)$, the pair $(A,B)$ is {\em complete} in $G$ if every vertex of $A$ is adjacent in $G$ to every vertex of $B$, and is {\em anticomplete} if $G$ if it is complete in $\overline G$.
	We also say that $A$ is {\em complete} (or {\em anticomplete}) to $B$ in $G$ if the pair $(A,B)$ is complete (or anticomplete, respectively) in $G$.
	A {\em blockade} in $G$ is a sequence $(B_1,\ldots,B_k)$ of disjoint subsets of $V(G)$;
	it is {\em complete} in $G$ if $(B_i,B_j)$ is complete in $G$ for all distinct $i,j\in[k]$, and {\em anticomplete} in $G$ if $(B_i,B_j)$ is anticomplete in $G$ for all distinct $i,j\in[k]$.
	The method of proof of \cref{thm:p5main} is via the following statement concerning high-$\rho$ complete pairs or complete blockades in $P_5$-free graphs:
	\begin{theorem}
		\label{thm:main}
		There exists $d\ge9$ such that every $P_5$-free graph $G$ with $\omega(G)\ge2$ contains either:
		\begin{itemize}
			\item a complete pair $(X,Y)$ with $\rho(X)\ge y^9\rho(G)$ and $\rho(Y)\ge (1-3y)\rho(G)$ for some $y\in(0,\frac14]$; or
			
			\item a complete blockade $(B_1,\ldots,B_k)$ for some $k\ge2$ such that $\rho(B_i)\ge k^{-d}\rho(G)$ for all $i\in[k]$.
		\end{itemize}
	\end{theorem}
	\begin{proof}
		[Proof of \cref{thm:p5main}, assuming \cref{thm:main}]
		We claim that $d\ge9$ given by \cref{thm:main} suffices.
		To see this, we proceed by induction on $w:=\omega(G)$; we may assume $w\ge2$.
		By the choice of $d$, $G$ contains either:
		\begin{itemize}
			\item a complete pair $(X,Y)$ with $\rho(X)\ge y^9\rho(G)$ and $\rho(Y)\ge(1-3y)\rho(G)$ for some $y\in(0,\frac14]$; or
			
			\item a complete blockade $(B_1,\ldots,B_k)$ for some $k\ge2$ such that $\rho(B_i)\ge k^{-d}\rho(G)$ for all $i\in[k]$.
		\end{itemize}
		
		If the first bullet holds then either $\omega(X)\le yw$ or $\omega(Y)\le(1-y)w$. If the former case holds then $1\le \omega(X)\le yw$ (since $\rho(X)>0$) and by induction $\rho(G)\le y^9\rho(X)\le y^9\omega(X)^d\le y^{9-d}w^d\le w^d$ by the choice of $d$;
		and if the latter case holds then $1\le\omega(Y)\le(1-y)w<w$ and so by induction
		\[(1-3y)\rho(G)\le \rho(Y)\le \omega(Y)^d\le (1-y)^dw^d\le (1-y)^9w^d\le (1-3y)w^d\]
		where the last inequality holds since (note that $y\in(0,\frac14]$)
		\[1-3y-(1-y)^9=y\sum_{j=0}^8(1-y)^i-3y
		\ge y\sum_{j=0}^8(3/4)^j-3y
		=y\cdot4(1-(3/4)^9)-3y\ge0.\]
		
		If the second bullet holds then $2\le k\le w$ and there exists $i\in[k]$ with $\omega(B_i)\le w/k$; and so by induction we have
		$\rho(G)\le k^d\rho(B_i)\le k^d(w/k)^d= w^d$.
		This proves \cref{thm:main}.
	\end{proof}
	
	Much of the rest of this paper is thus devoted to proving \cref{thm:main}.
	At a high level, the proof can be viewed as a combination of the density increment method in~\cite[Section 7]{density7} (under the name ``iterative sparsification'' but the argument there was for graphs with no induced  $\overline{P_5}$) and an adaptation to the $\rho$ setting of a decomposition argument for connected $P_5$-free graphs based on high-$\chi$ anticomplete pairs developed in~\cite[Section 4]{2024p5} (under the name ``terminal partitions'').
	The main advantage of working with $\rho$ compared to $\chi$ is that the former is much more amenable to counting vertices, which allows us to employ tools from~\cite{density7} and obtain a simplified variant of the method from~\cite{2024p5} for the Hall ratio.
	
	In what follows, for $\eps>0$, a graph $G$ is $\eps$-sparse if it has maximum degree at most $\eps\abs G$, is $(1-\eps)$-dense if the complement $\overline G$ of $G$ is $\eps$-sparse, and is $\eps$-restricted if it is $\eps$-sparse or $(1-\eps)$-dense.
	Since we are working with $\rho$, we may assume that $G$ is connected.
	Similar to a number of proofs of \erh{} type results, we begin with R\"odl's theorem~\cite{MR837962}:
	\begin{theorem}
		[R\"odl]
		\label{thm:rodl}
		For every $\eps\in(0,\frac12]$ and every graph $H$, there exists $\delta>0$ such that every $H$-free graph has an $\eps$-restricted induced subgraph $F$ with $\abs F\ge\delta\abs G$.
	\end{theorem}
	
	We will apply this result with $H=P_5$ and $\eps$ some appropriately small absolute constant.
	If the $\eps$-restricted induced subgraph $F$ turns out to be $(1-\eps)$-dense then we can follow the argument in~\cite[Section 7]{density7} and derive a complete blockade as in the second outcome of \cref{thm:main}.
	Here, for technical reasons involving $\rho$, to make the argument in this case work we need to employ the full \erh{} property of $P_5$ (actually, in its equivalent \dd polynomial R\"odl\ee{} form);
	this will be done in \cref{sec:incre}.
	
	Thus we may assume that $F$ is $\eps$-sparse. In this case we apply the following well-known result on linear-sized anticomplete pairs in sparse $P_5$-free graphs~\cite{MR3343757}:
	\begin{theorem}
		[Bousquet--Lagoutte--Thomass\'e]
		\label{thm:p5sparse}
		There exists $\eta\in(0,\frac14]$ such that every $\eta$-sparse $P_5$-free graph $F$ with $\abs F\ge2$ contains an anticomplete $(2,\eta\abs F)$-blockade.
	\end{theorem}
	
	This result is qualitatively optimal, in the sense that $F$ could just be a disjoint union of many copies of some graph and there is not much to improve even if one considers $\rho$ instead of size. Moreover, given how small $\eta$ is, iterating these anticomplete pairs would result in large exponents of $\alpha(G)$.
	At this point the situation is opposite to what happened in~\cite{density7}: there, the sparse case was assumed to be straightforward given \cref{thm:p5sparse} and much effort was devoted to the dense case; but here the dense case becomes simpler (provided the \erh{} property of $P_5$) while the sparse case appears nontrivial and a more global argument is required.
	
	Our approach to handle this is as follows. According to \cref{thm:p5sparse} and the above sketch of the dense case, we know that every $P_5$-free graph contains either a linear-sized anticomplete pair or a complete blockade satisfying the second outcome of \cref{thm:main} (see \cref{lem:incre}).
	Thus we can assume that every induced subgraph of the connected graph $G$ with linear $\rho$ does not contain a constant-dense linear-sized induced subgraph, and so contains a linear-sized (thus linear-$\rho$) anticomplete pair.
	Given such a ``locally sparse'' assumption, we can now adapt the ``decomposing along anticomplete pairs'' argument of~\cite[Section 4]{2024p5}.
	Our strategy will be using the adapted decomposition for $\rho$ (\cref{lem:anti}) to iteratively obtain anticomplete pairs $(A,B)$ in $G$ such that $\min(\rho(A),\rho(B))/\rho(G)$ is increasingly large.
	At the beginning, this ratio will be some small fixed constant resulting from \cref{thm:p5sparse};
	and after the first step, we will increase the ratio to $1-\eps$ for some fixed small $\eps>0$ for {\em all} connected $P_5$-free $G$ (\cref{lem:1step}).
	Provided that our desired complete pairs or blockades outcome as in \cref{thm:main} cannot occur, this allows us to upgrade the above ``locally sparse'' hypothesis for every such graph $G$ and every induced subgraph $F$ of $G$ with $\rho(F)\ge (1-\sqrt{\eps})\rho(G)$: $F$ contains an anticomplete pair of vertex subsets each with Hall ratio at least $(1-O(\eps))\rho(G)$.
	Now, starting from the second step, we will iteratively decrease $\eps$ (at each step for all $G$, see \cref{lem:incrho}) so that each connected graph $G$ in question will then be very close to being disconnected, in the sense that the minimal cutset separating the anticomplete pair $(A,B)$ in $G$ has Hall ratio much less than $\operatorname{poly}(\eps)\cdot\rho(G)$.
	But such a cutset is always nonempty; and so when $\eps$ becomes less than some negative power of $\rho(G)$, the first outcome of \cref{thm:main} automatically holds with $y=\eps$.
	
	Compared to the technique of~\cite[Section 4]{2024p5}, our decomposition argument in this case is simpler because of two reasons that also emphasise how amenable $\rho$ is to counting vertices.
	The first reason is, we can assume that there is always a linear-$\rho$ anticomplete pair in any linear-$\rho$ induced subgraphs given what we already know from the \erh{} property of $P_5$.
	In contrast, in~\cite{2024p5} obtaining a linear-$\chi$ anticomplete pair was nontrivial and could only be done assuming that there does not exist a ``locally dense'' and high-$\chi$ induced subgraph or a ``polynomial versus linear'' high-$\chi$ complete pair.
	The second reason is, under the assumption that the first outcome and the $k=2$ case of the second outcome (linear-$\rho$ complete pairs in other words) of \cref{thm:main} cannot occur, at each step of the aforementioned procedure we can use the \dd comb\ee{} lemma~\cite{MR4563865} (\cref{lem:comb}) to control the Hall ratio of the minimal cutset in question or deduce the second outcome of \cref{thm:main} (now with large $k$).
	On the other hand, such a scenario did not happen in~\cite{2024p5} because a suitable analogue of the \dd comb\ee{} lemma for $\chi$ has not been discovered.
	This finishes the summary of the proof ideas of \cref{thm:main}.
	
	\section{Density increment}
	\label{sec:incre}
	The main result of this section is as follows.
	\begin{lemma}
		\label{lem:incre}
		There exists $b\ge4$ such that for every $\eps\in(0,\frac12]$, every $P_5$-free graph $G$ with $\abs G\ge2$ contains either:
		\begin{itemize}
			\item an anticomplete pair $(A,B)$ with $\abs A,\abs B\ge2^{-b}\abs G$;
			
			\item a complete pair $(X,Y)$ with $\psi(X)\ge \eps^b\psi(G)$ and $\psi(Y)\ge2^{-b}\psi(G)$; or
			
			\item a complete blockade $(B_1,\ldots,B_k)$ in $G$ with $k\ge1/\eps$ and $\psi(B_i)\ge k^{-b}\psi(G)$ for all $i\in[k]$.
		\end{itemize}
	\end{lemma}
	
	In view of \cref{thm:p5sparse}, to prove \cref{lem:incre} we will deal with dense $P_5$-free graphs.
	As mentioned in \cref{sec:sketch}, the argument in this case is similar to the second part of the proof of the \erh{} property of $P_5$~\cite[Section  7]{density7}.
	To begin with, we require the following \dd polynomial R\"odl\ee{} property of $P_5$, which was shown to be equivalent to its \erh{} property~\cite{density7} (see also~\cite{bfp2024} for a proof of equivalence with $P_5$ replaced by any graph $H$).
	\begin{theorem}
		[Nguyen--Scott--Seymour]
		\label{thm:p5rodl}
		There exists $d\ge2$ such that for every $\eps\in(0,\frac12]$, every $P_5$-free graph $G$ contains an $\eps$-restricted induced subgraph on at least $\eps^d\abs G$ vertices.
	\end{theorem}
	We will refine this further for application. To do so, we applying \cref{thm:p5sparse} repeatedly to get:
	\begin{lemma}
		\label{thm:sparsep5}
		There exists $a\ge2$ such that for every $\eps\in(0,\frac12]$ and for
		every $\eps^a$-sparse $P_5$-free graph $G$ with $\abs G\ge\eps^{-a}$,
		there is an anticomplete blockade $(B_1,\ldots,B_k)$ in $G$ with $k\ge\eps^{-1}$ and $\abs{B_i}\ge\eps^a\abs G$ for all $i\in[k]$.
	\end{lemma}
	\begin{proof}
		We claim that $a:=2\log\eta^{-1}$ satisfies the theorem,
		where $\eta\in(0,\frac12)$ is given by \cref{thm:p5sparse}.
		To see this, let $G$ be an $\eps^a$-sparse $P_5$-free graph;
		and let $k\in\{0,1,\ldots,\ceil{\log\frac1\eps}\}$ be maximal such that there is an anticomplete $(2^k,\eta^{k}\abs G)$-blockade $(B_1,\ldots,B_{2^k})$ in $G$.
		
		If $k<\ceil{\log\frac1\eps}$, then
		for every $i\in[2^k]$, since 
		$\abs{B_i}\ge \eta^k\abs G\ge \eta^{-1} \cdot\eta^{2\log\frac1\eps}\abs G=\eta^{-1}\cdot \eps^a\abs G\ge\eta^{-1}$,
		$G[B_i]$ is $\eta$-sparse;
		and so there are disjoint and anticomplete $D_{2i-1},D_{2i}\subset B_i$ with $\abs{D_{2i-1}},\abs{D_{2i}}\ge \eta\abs{B_i}\ge \eta^{k+1}\abs G$.
		Then $(D_1,D_2,\ldots,D_{2^{k+1}})$ is an anticomplete $(2^{k+1},\eta^{k+1}\abs G)$-blockade in $G$ while $k+1\le\ceil{\log\frac1\eps}$,
		contrary to the maximality of $k$.
		
		Hence $k=\ceil{\log\frac1\eps}$.
		Now, since $2^k\ge 2^{\log\frac1\eps}=\eps^{-1}$
		and $\eta^k\ge\eta^{2\log\frac1\eps}=\eps^a$,
		there is an anticomplete $(2^k,\eta^k\abs G)$-blockade in $G$.
		This proves \cref{thm:sparsep5}.
	\end{proof}
	Combining \cref{thm:sparsep5,thm:p5rodl} gives the following refinement of \cref{thm:p5rodl}.
	\begin{lemma}
		\label{thm:p5antirodl}
		There exists $b\ge4$ such that for every $\eps\in(0,\frac12]$, every $P_5$-free graph $G$ contains either an $(1-\eps)$-dense induced subgraph on at least $\eps^b\abs G$ vertices,
		or an anticomplete blockade $(B_1,\ldots,B_k)$ with $k\ge\eps^{-1}$ and $\abs{B_i}\ge\eps^b\abs G$ for all $i\in[k]$.
	\end{lemma}
	\begin{proof}
		Let $a\ge2$ be given by \cref{thm:sparsep5}, and let $d\ge2$ be given by \cref{thm:p5rodl}.
		We claim that $b:=ad+a$ satisfies the theorem.
		To this end, let $G$ be a $P_5$-free graph;
		and we may assume $\abs G\ge\eps^{-b}$.
		By the choice of $d$, $G$ has an $\eps^a$-restricted induced subgraph $F$ with $\abs F\ge \eps^{ad}\abs G\ge\eps^{ad-b}=\eps^{-a}$.
		If $F$ is $(1-\eps^a)$-dense then we are done;
		so we may assume $F$ is $\eps^a$-sparse.
		By the choice of $a$,
		there is an anticomplete blockade $(B_1,\ldots,B_k)$ in $F$ with $k\ge\eps^{-1}$ and $\abs{B_i}\ge\eps^a\abs F\ge \eps^{a+ad}\abs G=\eps^b\abs G$ for all $i\in[k]$.
		This proves \cref{thm:p5antirodl}.
	\end{proof}
	
	Now we begin handling dense $P_5$-free graphs.
	In what follows we utilise the simple fact that $\psi(G)\le\max_F\psi(F)$ where $F$ ranges through all components of~$G$.
	\begin{lemma}
		\label{lem:p5coml}
		There exists $b\ge4$ such that for every $x,y\in(0,\frac12]$, every $(1-x)$-dense $P_5$-free graph $G$ with $\abs G\ge y^{-b^2}$ contains either:
		\begin{itemize}
			\item a $(1-y^b)$-dense induced subgraph with at least $y^{b^2}\abs G$ vertices; or
			
			\item a complete pair $(X,Y)$ such that $\psi(X)\ge y^{b^2}\psi(G)$ and $\abs Y\ge(1-x-y)\abs G$.
		\end{itemize}
	\end{lemma}
	\begin{proof}
		We claim that $b\ge4$ given by \cref{thm:p5antirodl} satisfies the lemma.
		To see this,
		let $G$ be a $(1-x)$-dense $P_5$-free graph;
		and let $\eps:=y^b$.
		By the choice of $b$, $G$ contains either:
		\begin{itemize}
			\item a $(1-y^b)$-dense induced subgraph on at least $y^{b^2}\abs G$ vertices; or
			
			\item an anticomplete blockade $(B_1,\ldots,B_k)$ with $k\ge y^{-b}$ and $\abs{B_i}\ge y^{b^2}\abs G$ for all $i\in[k]$.
		\end{itemize}
		
		If the first bullet holds then we are done; so let us assume that the second bullet holds.
		We may also assume $k=\ceil{y^{-b}}\le2y^{-b}$ and $\abs{B_i}=\ceil{y^{b^2}\abs G}\le 2y^{b^2}\abs G$ for all $i\in[k]$;
		thus for $B:=B_1\cup\cdots\cup B_k$ we have $\abs B\le k\cdot 2y^{b^2}\abs G\le 4y^{b^2-b}\abs G\le y^2\abs G\le\frac12 y\abs G$.
		For each $i\in[k]$, there exists $A_i\subset B_i$ such that $G[A_i]$ is connected and
		$$\psi(A_i)\ge\psi(B_i)\ge \abs{B_i}/\alpha(G)\ge y^{b^2}\psi(G).$$
		Now, since $G$ is $P_5$-free, no vertex in $G\setminus B$ is mixed on at least two of $A_1,\ldots,A_k$;
		and so there exists $i\in[k]$ for which at most $\abs{G\setminus B}/k\le \frac12y\abs G$ vertices in $G\setminus B$ are mixed on $A_i$.
		Thus, since $G$ is $(1-x)$-dense and $\abs B\le \frac12y\abs G$, there are at most $(x+y)\abs G$ vertices in $G$ with a nonneighbour in $A_i$.
		Hence there are at least $(1-x-y)\abs G$ vertices in $G$ complete to $A_i$.
		Letting $Y$ be the set of such vertices and $X:=A_i$ then verifies the second outcome of the lemma.
		This proves \cref{lem:p5coml}.
	\end{proof}
	We remark that in the above proof, in general we do not know how large each $A_i$ is compared to $B_i$. Thus, if one only applied~\cite[Theorem 5.1]{density7} (in the complement) to obtain $(B_1,\ldots,B_k)$ with the weaker property that each pair $(B_i,B_j)$ is either anticomplete or \dd dense\ee{}, then the density increment step combining $A_1,\ldots,A_k$ as in~\cite[Section 7]{density7} likely would not work.
	This highlights the strength of \cref{thm:p5antirodl} with two separate outcomes (an anticomplete blockade or a dense induced subgraph) and in turn explains why the \erh{} (or \dd polynomial R\"odl\ee) property of $P_5$ plays an important role in this case.

	Now, iterating \cref{lem:p5coml} with $x=y$ gives:
	\begin{lemma}
		\label{lem:p5compl}
		Let $b\ge4$ be given by \cref{lem:p5coml}. Then for every $y\in(0,\frac14]$, every $(1-y)$-dense $P_5$-free graph $G$ with $\abs G\ge y^{-2b^2}$ contains either:
		\begin{itemize}
			\item a $(1-y^b)$-dense induced subgraph on at least $y^{2b^2}\abs G$ vertices; or
			
			\item a complete blockade $(B_1,\ldots,B_k)$ with $k\ge1/y$ and $\psi(B_i)\ge y^{2b^2}\psi(G)$ for all $i\in[k]$.
		\end{itemize}
	\end{lemma}
	\begin{proof}
		Let $k\ge0$ be maximal such that there is a complete blockade $(B_0,B_1,\ldots,B_k)$ in $G$ where $\abs{B_k}\ge(1-2y)^k\abs G$ and $\psi(B_i)\ge y^{2b^2}\psi(G)$ for all $i\in[k]$.
		If $k\ge1/y$ then the second outcome of the lemma holds and we are done; so let us suppose that $k<1/y$.
		Then, since $y\le\frac14$, we have
		\[\abs{B_k}\ge(1-2y)^{k}\abs G
		\ge 4^{-2yk}\abs G\ge 2^{-4}\abs G.\]
		By the choice of $b$, $G[B_k]$ contains either:
		\begin{itemize}
			\item a $(1-y^b)$-dense induced subgraph on at least $y^{b^2}\abs{B_k}$ vertices; or
			
			\item a complete pair $(X,Y)$ in $G[B_k]$ with $\psi(X)\ge y^{b^2}\psi(B_k)$ and $\abs Y\ge(1-2y)\abs{B_k}$.
		\end{itemize}
		
		The first bullet fails since $y^{b^2}\abs{B_k}\ge2^{-4}y^{b^2}\abs{B_k}\ge y^{2b^2}\abs G$;
		thus the second bullet holds.
		Then, since $\psi(X)\ge y^{b^2}\psi(B_k)\ge y^{b^2}\abs{B_k}/\alpha(G)\ge 2^{-4}y^{b^2}\psi(G)\ge y^{2b^2}\psi(G)$
		and $\abs Y\ge(1-2y)\abs{B_k}\ge(1-2y)^{k+1}\abs G$,
		the blockade $(B_0,B_1,\ldots,B_{k-1},X,Y)$ violates the maximality of $k$.
		This proves \cref{lem:p5compl}.
	\end{proof}
	We are now in a position to perform density increment in the dense case, via the following lemma.
	\begin{lemma}
		\label{lem:dense}
		There exists $a\ge4$ such that for every $\eps \in(0,\frac12]$ and $\eta\in(0,\frac14]$, every $(1-\eta)$-dense $P_5$-free graph $G$ with $\abs G\ge2$ contains either:
		\begin{itemize}
			\item a complete pair $(X,Y)$ with $\psi(X)\ge\eps^a\psi(G)$ and $\psi(Y)\ge \frac14\psi(G)$; or
			
			\item a complete blockade $(B_1,\ldots,B_k)$ with $k\ge1/\eps$ and $\psi(B_i)\ge k^{-a}\psi(G)$ for all $i\in[k]$.
		\end{itemize}
	\end{lemma}
	\begin{proof}
		Let $b\ge4$ be given by \cref{lem:p5coml};
		we claim that $a:=4b^3+6b^2$ suffices.
		To this end, let $\xi:=\psi(G)^{-1/a}$.
		We may assume $\eps>\xi$, for otherwise the first outcome of the lemma holds (with $\abs X=1$) since $G$ is $(1-\eta)$-dense and $\abs G\ge2$.
		Hence $\abs G\ge \psi(G)=\xi^{-a}>\eps^{-b^2}$ by the choice of $a$.
		Thus, by \cref{lem:p5coml} with $x=\eta$ and $y=\eps$, $G$ contains either:
		\begin{itemize}
			\item a $(1-\eps^b)$-dense induced subgraph with at least $\eps^{b^2}\abs G$ vertices; or
			
			\item a complete pair $(X,Y)$ with $\psi(X)\ge \eps^{b^2}\psi(G)\ge \eps^a\psi(G)$ and $\abs Y\ge (1-\eta-\eps)\abs G\ge \frac14\abs G$.
		\end{itemize}
		
		If the second bullet holds then the first outcome of the lemma holds and we are done.
		Thus, we may assume that the first bullet holds.
		Therefore, there exists $y\in[\xi^{2b},\eps^b]$ minimal such that $G$ has a $(1-y)$-dense induced subgraph $F$ with $\abs F\ge y^{3b}\abs G$.
		If $y\le \xi^{2}$, then since $\abs F\ge y^{3b}\abs G\ge \xi^{6b^2-a}>\xi^{-2}$ by the choice of $a$, we have $\omega(F)\ge\frac12\xi^{-2}>\xi^{-1}>\eps^{-1}$; and so the second outcome of the lemma holds with $k=\omega(F)$.
		Hence, we may assume $y>\xi^{2}$.
		Now, because $\abs F\ge y^{3b}\abs G\ge y^{3b}\xi^{-a}\ge y^{3b-a/(2b)}= y^{-2b^2}$,
		\cref{lem:p5compl} implies that $F$ contains either:
		\begin{itemize}
			\item a $(1-y^b)$-dense induced subgraph on at least $y^{2b^2}\abs F$ vertices; or
			
			\item a complete blockade $(B_1,\ldots,B_k)$ in $F$ with $k\ge1/y\ge1/\eps$ and $\psi(B_i)\ge y^{2b^2}\psi(F)$ for all $i\in[k]$.
		\end{itemize}
		
		Since $y^b>\xi^{2b}$ and $y^{2b^2}\abs F\ge y^{2b^2+3b}\abs G\ge y^{3b^2}\abs G$, the first bullet cannot hold by the minimality of $y$.
		Thus the second bullet holds;
		and so the second outcome of the lemma holds since $y^{2b^2}\psi(F)\ge y^{2b^2+3b}\psi(G)\ge y^a\psi(G)\ge k^{-a}\psi(G)$.
		This proves \cref{lem:dense}.
	\end{proof}
	It now suffices to remove the dense hypothesis by means of R\"odl's theorem \ref{thm:rodl} to prove \cref{lem:incre}, which we restate here for convenience:
	\begin{lemma}
		\label{lem:p5}
		There exists $b\ge4$ such that for every $\eps\in(0,\frac12]$, every $P_5$-free graph $G$ with $\abs G\ge2$ contains either:
		\begin{itemize}
			\item an anticomplete pair $(A,B)$ with $\abs A,\abs B\ge2^{-b}\abs G$;
			
			\item a complete pair $(X,Y)$ with $\psi(X)\ge \eps^b\psi(G)$ and $\psi(Y)\ge2^{-b}\psi(G)$; or
			
			\item a complete blockade $(B_1,\ldots,B_k)$ in $G$ with $k\ge1/\eps$ and $\psi(B_i)\ge k^{-b}\psi(G)$ for all $i\in[k]$.
		\end{itemize}
	\end{lemma}
	\begin{proof}
		Let $\eta\in(0,\frac14]$ be given by \cref{thm:p5sparse}, 
		let $\delta>0$ be given by \cref{thm:rodl} with $\eps=\eta$, and let $a\ge4$ be given by \cref{lem:dense}.
		We claim that $b:=a+\log\delta^{-1}$ suffices.
		To see this, the choice of $\delta$ gives an $\eta$-restricted induced subgraph of $G$ with $\abs J\ge \delta\abs G$.
		If $J$ is $\eta$-sparse then it contains an anticomplete pair $(A,B)$ with $\abs A,\abs B\ge \eta\abs J\ge \eta\delta\abs G\ge 2^{-b}\abs G$ by the choice of $b$; but this satisfies the first outcome of the lemma, a contradiction.
		So we may assume $J$ is $(1-\eta)$-dense.
		By the choice of $a$, $J$ contains either:
		\begin{itemize}
			\item a complete pair $(X,Y)$ with $\psi(X)\ge \eps^a\psi(J)$ and $\psi(Y)\ge\frac14\psi(J)$; or
			
			\item a complete blockade $(B_1,\ldots,B_k)$ with $k\ge1/\eps$ and $\psi(B_i)\ge k^{-a}\psi(J)$ for all $i\in[k]$.
		\end{itemize}
		
		If the first bullet holds then the second outcome of the lemma holds since $\eps^a\psi(J)\ge \eps^{a}\delta\cdot\psi(G)\ge \eps^{a+\log\delta^{-1}}\psi(G)\ge \eps^{b}\psi(G)$
		and $\frac14\psi(J)\ge\frac14\delta\cdot\psi(G)\ge \eta\delta\cdot\psi(G)\ge 2^{-b}\psi(G)$ by the choice of $b$;
		and if the second bullet holds then the third outcome of the lemma holds since $k^{-a}\psi(J)\ge k^{-a-\log\delta^{-1}}\psi(G)\ge k^{-b}\psi(G)$.
		This proves \cref{lem:p5}.
	\end{proof}
	
	\section{Obtaining increasingly high-$\rho$ anticomplete pairs}
	\label{sec:anti}
	
	In this section we handle the locally sparse case and eventually complete the proof of \cref{thm:main}.
	Before carrying out, we would like to recall the following simple properties of $\rho$ for every graph $G$:
	\begin{itemize}
		\item $\rho(\{v\})=1$ for all $v\in V(G)$; and
		
		\item $\max(\rho(A),\rho(B))\le\rho(A\cup B)\le\rho(A)+\rho(B)$ for all $A,B\subset V(G)$, and equality holds for the left-hand side inequality if $A$ is anticomplete to $B$ in $G$.
	\end{itemize} 
	
	As mentioned in \cref{sec:sketch}, we require the following corollary of the \dd comb\ee{} lemma~\cite[Lemma 2.1]{MR4563865}.
	\begin{lemma}
		[Chudnovsky--Scott--Seymour--Spirkl]
		\label{lem:comb}
		Let $G$ be a graph with nonempty and disjoint $A,B\subset V(G)$, such that each vertex in $A$ has at most $\Delta>0$ neighbours in $B$ and every vertex in $B$ has a neighbour in $A$. Then for every $\Gamma>0$, either:
		\begin{itemize}
			\item $\abs B<20\sqrt{\Gamma\Delta}$; or
			
			\item for some integer $k\ge1$, there are $k$ vertices $a_1,\ldots,a_k$ in $A$ and $k$ disjoint subsets $B_1,\ldots,B_k$ of $B$, such that for every $i\in[k]$, $\abs{B_i}\ge\Gamma/k^2$, and $a_i$ is adjacent in $G$ to every vertex in $B_i$ and to no vertex in $B_j$ for all $j\in[k]\setminus\{i\}$.
		\end{itemize}
	\end{lemma}
	
	For $q\ge p>0$, we say that $G$ is {\em $(p,q)$-sparse} if for every induced subgraph $F$ of $G$ with $\rho(F)\ge q$,
	there is an anticomplete pair $(X,Y)$ in $F$ with $\rho(X),\rho(Y)\ge p$.
	The following lemma investigates $(p,q)$-sparse $P_5$-free graphs $G$ with $q$ not too large compared to $\rho(G)$.
	\begin{lemma}
		\label{lem:anti}
		Let $\eps\in(0,\frac12]$, let $G$ be a $P_5$-free graph, and let $0<p\le q\le(1-\eps^2)\rho(G)$. 
		If $G$ is $(p,q)$-sparse, then it contains either:
		\begin{itemize}
			\item an anticomplete pair $(A,B)$ with $\rho(A)\ge q-2\eps^8\rho(G)$ and $\rho(B)\ge(1-\eps^2)\rho(G)$;
			
			\item a complete pair $(X,Y)$ with $\rho(X)\ge \eps^8\rho(G)$ and $\rho(Y)\ge p$; or
			
			\item a complete blockade $(B_1,\ldots,B_k)$ with $k\ge1/\eps$ and $\rho(B_i)\ge k^{-8}\rho(G)$ for all $i\in[k]$.
		\end{itemize}
	\end{lemma}
	(In application we will let $p$ vary and always fix $q=(1-\eps^2)\rho(G)$, but for better clarity we decided to make $q$ an independent variable in this lemma.)
	\begin{proof}
		The proof makes use of high-$\rho$ anticomplete pairs to decompose $G$ in a similar way to \dd terminal partitions\ee{} developed in~\cite[Section 4]{2024p5}.
		We may assume that $G$ is connected and $\psi(G)=\rho(G)$.
		We say that $S\subset V(G)$ {\em separates} nonempty and disjoint $B_1,\ldots,B_k\subset V(G)\setminus S$ in $G$ if $G\setminus S$ has no path between $B_i$  and $B_j$ for all distinct $i,j\in[k]$; and $S$ is a {\em cutset} of $G$ if $k\ge2$.
		
		For the proof, assume that the last two outcomes do not hold.
		We begin with:
		\begin{claim}
			\label{claim:sparse}
			Every connected induced subgraph $F$ of $G$ with $\rho(F)\ge q$ contains a minimal nonempty cutset separating two vertex subsets $A,B$ with $\rho(A)\ge \max(p,q-2\eps^8\rho(G))$ and $\rho(B)\ge p$.  
		\end{claim}
		\begin{subproof}
			Since $G$ is $(p,q)$-sparse, $F$ contains an anticomplete pair $(A,B)$ with $\rho(A)\ge\rho(B)\ge p$; and we may assume that $F[A],F[B]$ are connected.
			Among all such anticomplete pairs $(A,B)$, choose one such that $\rho(A)$ is as large as possible.
			Since $F$ is connected, there is a minimal nonempty cutset $S$ separating $A,B$ in $F$.
			Thus, because $F$ is $P_5$-free, every vertex in $S$ is complete to $A$ or to $B$ in $F$.
			Since the second outcome of the lemma does not hold, it follows that $\rho(S)\le 2\eps^8\rho(G)$; and so the maximality of $\rho(A)$ implies that
			$\rho(A)=\rho(F\setminus S)\ge \rho(F)-2\eps^8\rho(G)\ge q-2\eps^8\rho(G)$.
			This proves \cref{claim:sparse}.
		\end{subproof}
		
		By \cref{claim:sparse} with $F=G$, $G$ has a minimal nonempty cutset $S$ separating two vertex subsets $A,B$ of $G$ with $\rho(B)\ge p$ and
		\[\rho(A)\ge(1-2\eps^8)\rho(G)\ge \max(p,q-2\eps^8\rho(G)).\]
		
		Therefore, there exists $k\ge1$ maximal such that there is a partition $(A,D,B_1,\ldots,B_k,E)$ of $V(G)$ into nonempty subsets satisfying:
		\begin{itemize}
			\item $D$ is a cutset separating $A,B_1,\ldots,B_k,E$ in $G$;
			
			\item $G[A]$ and $G[B_1],\ldots,G[B_k]$ are connected;
			
			\item every vertex in $D$ has a neighbour in $B_1\cup\cdots\cup B_k$;
			
			\item $\rho(A)\ge \max(p,q-2\eps^8\rho(G))$, $\rho(E)<p$, and $\rho(B_i)\ge p$ for all $i\in[k]$.
		\end{itemize}
		(When $k=1$ one can take $D=S$ and $B_1=B$.)
		The main consequence of the maximality of $k$ is that $\rho(A)$ cannot be too large:
		\begin{claim}
			\label{claim:smalla}
			$\rho(A)<q$.
		\end{claim}
		\begin{subproof}
			Suppose not.
			Then by \cref{claim:sparse}, $G[A]$ contains a minimal nonempty cutset $S'$ separating two vertex subsets $A',B'$ with $\rho(A')\ge\max(p,q-2\eps^8\rho(G))$ and $\rho(B')\ge p$.
			Among the components of $G[A\setminus(A'\cup S')]$, let $G[B_{k+1}],\ldots,G[B_{k+r}]$ be those with Hall ratio at least $p$ (thus $r\ge1$), and let $E':=A\setminus(A'\cup S'\cup(B_{k+1}\cup\cdots\cup B_{k+r}))$; then $\rho(E')<p$.
			Then
			\[(A',D\cup S',B_1,\ldots,B_k,B_{k+1},\ldots,B_{k+r},E\cup E')\]
			would be a partition of $V(G)$ violating the maximality of $k$.
			This proves \cref{claim:smalla}.
		\end{subproof}
		
		Now, let $S$ be the set of vertices in $D$ with a neighbour in~$A$.
		We use the $P_5$-freeness of $G$ to bound $\rho(S)$, as follows:
		\begin{claim}
			\label{claim:smalls}
			$\rho(S)\le\eps^2\rho(G)$.
		\end{claim}
		\begin{subproof}
			Suppose not. Let $S'$ be the set of vertices in $S$ mixed on $A$; then $S\setminus S'$ is complete to $A$ and so $\rho(S\setminus S')\le\eps^8\rho(G)$ since the second outcome of the lemma fails.
			Hence
			\[\rho(S')\ge\rho(S)-\rho(S\setminus S')\ge(\eps^2-\eps^8)\rho(G)\ge\frac{15}{16}\eps^2\rho(G).\]
			Let $S_0\subset S'$ be such that $\psi(S_0)=\rho(S')\ge\frac{15}{16}\eps^2\rho(G)$.
			Since $G$ is $P_5$-free, and since every $v\in S_0\subset S$ is mixed on $A$, it follows that $v$ is complete or anticomplete to each of $B_1,\ldots,B_k$.
			Thus, because every vertex in $S$ has a neighbour in $B_1\cup\cdots\cup B_k$, every vertex in $S_0$ is complete in $G$ to at least one of $B_1,\ldots,B_k$.
			For each $i\in[k]$, let $S_i$ be the set of vertices in $S_0$ complete to $B_i$ in $G$; then since the second outcome of the lemma fails, we have $\psi(S_i)\le\rho(S_i)\le\eps^8\rho(G)\le\frac{16}{15}\eps^6\cdot\psi(S_0)$,
			and so $\abs{S_i}\le\frac{16}{15}\eps^6\abs{S_0}$.
			Also, $S_0=S_1\cup\cdots\cup S_k$.
			Let $b_i\in B_i$ for each $i\in[k]$.
			Then by remunerating $[k]$ if necessary, by \cref{lem:comb} with $\{b_1,\ldots,b_k\}$ in place of $A$, $S_0$ in place of $B$, $\Delta=\frac{16}{15}\eps^6\abs{S_0}$, and $\Gamma=\frac3{1280}\eps^{-6}\abs{S_0}$, we see that either:
			\begin{itemize}
				\item $\abs{S_0}<20\sqrt{\Gamma\Delta}=\abs{S_0}$; or
				
				\item there exist $\ell\in[k]$ and $\ell$ disjoint subsets $X_1,\ldots,X_\ell$ of $S$ satisfying for every $i\in[\ell]$, $\abs{X_i}\ge \Gamma/\ell^2$, $X_i\subset S_i$, and $b_i$ is adjacent in $G$ to every vertex in $X_i$ and to no vertex in $X_j$ for all $j\in[\ell]\setminus\{i\}$.
			\end{itemize}
			
			The first bullet cannot hold, so the second bullet holds. Then for every $i\in[\ell]$, we have
			\[\frac{16}{15}\eps^6\abs {S_0}\ge\abs{S_i}\ge\abs{X_i}\ge\Gamma/\ell^2=\frac3{1280}\eps^{-6}\ell^{-2}\abs{S_0}\]
			so $\ell\ge\frac3{64}\eps^{-6}\ge1/\eps$; and furthermore
			\[\rho(X_i)\ge\psi(X_i) \ge\frac3{1280}\eps^{-6}\ell^{-2}\psi(S_0)
			\ge\frac9{4096}\eps^{-3}\ell^{-2}\rho(G)
			\ge \frac9{512}\ell^{-2}\rho(G)
			\ge \ell^{-8}\rho(G).\]
			
			Now, suppose that there are distinct $i,j\in[\ell]$ for which there are $u\in X_i$ and $v\in X_j$ nonadjacent in $G$.
			Let $P$ be a path of $G$ with endpoints $u,v$ and interior in $A$, which exists since $G[A]$ is connected and each of $u,v$ has a neighbour in $A$.
			Then $b_i\text-P\text-b_j$ would be an induced path in $G$ of length at least four, contrary to the $P_5$-freeness of $G$.
			Hence $(X_1,\ldots,X_\ell)$ is a complete blockade in $G$ with $\ell\ge1/\eps$ and $\rho(X_i)\ge\ell^{-8}\rho(G)$ for all $i\in[\ell]$; and this satisfies the third outcome of the lemma, a contradiction.
			This proves \cref{claim:smalls}.
		\end{subproof}
		
		Now, let $B:=B_1\cup\cdots\cup B_k$; then \cref{claim:smalls,claim:smalla} together imply that
		\[\rho(A\cup(D\setminus S)\cup B\cup E)
		\ge \rho(G)-\rho(S)>(1-\eps^2)\rho(G)\ge q>\rho(A).\]
		Thus, since $A$ is anticomplete to $(D\setminus S)\cup B\cup E$, \cref{claim:smalla} yields
		\[\rho((D\setminus S)\cup B\cup E)>(1-\eps^2)\rho(G)\]
		and so $(A,(D\setminus S)\cup B\cup E)$ satisfies the first outcome of the lemma.
		This proves \cref{lem:anti}.
	\end{proof}
	We now turn the first outcome of \cref{lem:p5} (for all $P_5$-free graphs $G$) into an anticomplete pair with very high Hall ratio.
	An interested reader may note that the following result is sufficient to deduce the off-diagonal \erh{} property of $P_5$:
	\begin{lemma}
		\label{lem:1step}
		There exists $a\ge8$ such that for every $\eps\in(0,\frac12]$, every $P_5$-free graph $G$ with $\omega(G)\ge2$ contains either:
		\begin{itemize}
			\item an anticomplete pair $(A,B)$ with $\rho(A),\rho(B)\ge(1-\eps)\rho(G)$;
			
			\item a complete pair $(X,Y)$ with $\rho(X)\ge\eps^a\rho(G)$ and $\rho(Y)\ge2^{-a}\rho(G)$; or
			
			\item a complete blockade $(B_1,\ldots,B_k)$ with $k\ge1/\eps$ and $\rho(B_i)\ge k^{-a}\rho(G)$ for all $i\in[k]$.
		\end{itemize}
	\end{lemma}
	\begin{proof}
		Let $b\ge1$ be given by \cref{lem:p5}; we claim that $a:=\max(b+1,8)$ suffices.
		To this end, let $p:=2^{-a}\rho(G)$ and $q:=(1-\eps^2)\rho(G)$, and suppose that none of the outcomes holds; we claim:
		\begin{claim}
			\label{claim:sp}
			$G$ is $(p,q)$-sparse.
		\end{claim}
		\begin{subproof}
			Let $F$ be an induced subgraph of $G$ with $\rho(F)\ge q= (1-\eps^2)\rho(G)\ge\frac34\rho(G)\ge\frac34\omega(G)\ge\frac32$, and let $J$ be an induced subgraph of $F$ with $\psi(J)=\rho(F)$;
			hence $\abs J\ge \psi(J)=\rho(F)\ge\frac32$ and so $\abs J\ge2$.
			By the choice of $b$, $J$ (and so $F$) contains either:
			\begin{itemize}
				\item an anticomplete pair $(A,B)$ with $\abs A,\abs B\ge 2^{-b}\abs J$;
				
				\item a complete pair $(X,Y)$ with $\psi(X)\ge\eps^b\psi(J)$ and $\psi(Y)\ge 2^{-b}\psi(J)$; or
				
				\item a complete blockade $(B_1,\ldots,B_k)$ with $k\ge1/\eps$ and $\psi(B_i)\ge k^{-b}\psi(J)$ for all $i\in[k]$.
			\end{itemize}
			
			The last two bullets do not hold because the last two outcomes of the theorem fail, $\psi(J)=\rho(F)\ge\frac12\rho(G)\ge\eps\cdot\rho(G)$, and $a\ge b+1$ by the choice of $a$.
			Thus the first bullet holds; and so $\rho(A),\rho(B)\ge 2^{-b}\psi(J)\ge2^{-b-1}\rho(G)\ge 2^{-a}\rho(G)$.
			This proves \cref{claim:sp}.
		\end{subproof}
		
		Now, by \cref{claim:sp,lem:anti}, $G$ contains either:
		\begin{itemize}
			\item an anticomplete pair $(A,B)$ with $\rho(A)\ge q-2\eps^8\rho(G)=(1-\eps^2-2\eps^8)\rho(G)\ge(1-\eps)\rho(G)$ and $\rho(B)\ge(1-\eps^2)\rho(G)$;
			
			\item a complete pair $(X,Y)$ with $\rho(X)\ge\eps^8\rho(G)$ and $\rho(Y)\ge p=2^{-a}\rho(G)$; or
			
			\item a complete blockade $(B_1,\ldots,B_k)$ with $k\ge1/\eps$ and $\rho(B_i)\ge k^{-8}\rho(G)$ for all $i\in[k]$.
		\end{itemize}
		Since $a\ge 8$ by the choice of $a$,
		one of the outcomes of the lemma holds, a contradiction.
		This proves \cref{lem:1step}.
	\end{proof}
	
	In what follows, let $\phi_0:=1$; and
	for every integer $s\ge1$, let $\phi_s:=\prod_{i=1}^{s}(1-2^{-2^{i+1}})$ if $s\ge1$,
	and let $\phi_{r,s}:=\phi_s/\phi_r=\prod_{r<i\le s}(1-2^{-2^{i+1}})$ for every integer $r$ with $0\le r\le s$ (so $\phi_{0,s}=\phi_s$).
	The following lemma showcases the process of extracting anticomplete pairs with increasingly high-$\rho$ (actually, arbitrarily close to $\rho(G)$) from $P_5$-free graphs $G$.
	\begin{lemma}
		\label{lem:incrho}
		Let $a\ge8$ be given by \cref{lem:1step}. Then for every integer $s\ge0$, every $P_5$-free graph $G$ with $\omega(G)\ge2$ contains either:
		\begin{itemize}
			\item an anticomplete pair $(A,B)$ with $\rho(A),\rho(B)\ge(1-2^{1-2^{s+1}})\rho(G)$;
			
			\item a complete pair $(X,Y)$ with
			$$\rho(X)\ge \phi_{r,s}\cdot2^{-2^{r+3}}\rho(G)
			\qquad\text{and}\qquad
			\rho(Y)\ge \phi_{r,s}\cdot\left(1-\frac52\cdot2^{-2^{r}}\right)\rho(G)$$
			for some integer $r$ with $1\le r\le s$ (so $s\ge1$); or
			
			\item a complete blockade $(B_1,\ldots,B_k)$ with $k\ge2$ and $\rho(B_i)\ge \phi_s\cdot k^{-a}\rho(G)$ for all $i\in[k]$.
		\end{itemize}
	\end{lemma}
	\begin{proof}
		We proceed by induction on $s\ge0$. The case $s=0$ follows from \cref{lem:1step} with $\eps=\frac12$, the second and third outcomes of which verify the third outcome of the lemma.
		Now, for $s\ge0$, we will prove the lemma for $s+1$ assuming that it is true for $s$ (and for all $P_5$-free graphs $G$ with $\omega(G)\ge2$).
		To this end, let $G$ be a $P_5$-free graph; and assume that the last two outcomes of the lemma (for $s+1$ in place of $s$) do not hold. Let $\eps:=2^{-2^{s+1}}$, $p:=(1-\frac52\eps)\rho(G)$, and $q:=(1-\eps^{2})\rho(G)$.
		We claim that:
		\begin{claim}
			\label{claim:incrho}
			$G$ is $(p,q)$-sparse.
		\end{claim}
		\begin{subproof}
			Let $F$ be an induced subgraph of $G$ with $\rho(F)\ge q=(1-\eps^2)\rho(G)=(1-2^{-2^{s+2}})\rho(G)$;
			then $\rho(F)>\frac12\rho(G)\ge\frac12\omega(G)\ge1$ and so $\omega(F)\ge2$.
			By induction for $s$ applied to $F$, $F$ contains either:
			\begin{itemize}
				\item an anticomplete pair $(A,B)$ with $\rho(A),\rho(B)\ge(1-2^{1-2^{s+1}})\rho(F)=(1-2\eps)\rho(F)$;
				
				\item a complete pair $(X,Y)$ with $$\rho(X)\ge\phi_{r,s}\cdot2^{-2^{r+3}}\rho(F)
				\qquad\text{and}\qquad
				\rho(Y)\ge\phi_{r,s}\cdot\left(1-\frac52\cdot2^{-2^{r}}\right)\rho(F)$$
				for some integer $r$ with $1\le r\le s$; or
				
				\item a complete blockade $(B_1,\ldots,B_k)$ with $k\ge2$ and $\rho(B_i)\ge\phi_s\cdot k^{-a}\rho(F)$ for all $i\in[k]$.
			\end{itemize}
			
			Because
			$$\phi_{r,s}\cdot\rho(F)\ge \phi_{r,s}\cdot(1-2^{-2^{s+2}})\rho(G)=\phi_{r,s+1}\cdot\rho(G)$$
			for all integers $r$ with $0\le r\le s$,
			if the second or third bullet holds then one of the second or third outcome of the lemma holds (respectively) for $s+1$, a contradiction.
			Thus the first bullet holds; and this proves \cref{claim:incrho}
			because
			\[(1-2\eps)\rho(F)\ge(1-2\eps)(1-\eps^2)\rho(G)\ge\left(1-\frac52\eps\right)\rho(G)=p.
			\qedhere\]
		\end{subproof}
		
		Now, by \cref{lem:anti}, $G$ contains either:
		\begin{itemize}
			\item an anticomplete pair $(A,B)$ with $\rho(A)\ge q-2\eps^8\rho(G)$ and $\rho(B)\ge(1-\eps^2)\rho(G)$;
			
			\item a complete pair $(X,Y)$ with $\rho(X)\ge \eps^8\rho(G)$ and $\rho(Y)\ge p$; or
			
			\item a complete blockade $(B_1,\ldots,B_k)$ with $k\ge2$ and $\rho(B_i)\ge k^{-8}\rho(G)$ for all $i\in[k]$.
		\end{itemize}
		
		If the second bullet holds then \[\rho(X)\ge\eps^8\rho(G)=2^{-2^{s+4}}\rho(G)
		\qquad\text{and}\qquad
		\rho(Y)\ge p=\left(1-\frac52\cdot 2^{-2^{s+1}}\right)\rho(G)\]
		and so the second outcome of the lemma holds for $s+1$ (with $r=s+1$), a contradiction.
		If the third bullet holds then the last outcome of the lemma holds, a contradiction.
		
		Hence the first bullet holds; then $\rho(B)\ge(1-\eps^2)\rho(G)=(1-2^{-2^{s+2}})\rho(G)$ and
		\[\rho(A)\ge q-2\eps^8\rho(G)\ge (1-\eps^2-2\eps^8)\rho(G)\ge(1-2\eps^2)\rho(G)
		=(1-2^{1-2^{s+2}})\rho(G)\]
		and so the first outcome of the lemma holds for $s+1$. This proves \cref{lem:incrho}.
	\end{proof}
	
	We are now ready to prove \cref{thm:main}, which we restate here for convenience:
	\begin{theorem}
		\label{thm:niam}
		There exists $d\ge9$ such that every $P_5$-free graph $G$ with $\omega(G)\ge2$ contains either:
		\begin{itemize}
			\item a complete pair $(X,Y)$ with $\rho(X)\ge y^9\rho(G)$ and $\rho(Y)\ge (1-3y)\rho(G)$ for some $y\in(0,\frac14]$; or
			
			\item a complete blockade $(B_1,\ldots,B_k)$ with $k\ge2$ and $\rho(B_i)\ge k^{-d}\rho(G)$ for all $i\in[k]$.
		\end{itemize}
	\end{theorem}
	\begin{proof}
		Let $a\ge8$ be given by \cref{lem:1step}. We claim that $d:=a+1$ suffices.
		To see this, assume that $G$ is connected, and suppose that both outcomes do not hold.
		Let $w:=\rho(G)\ge2$; we claim that:
		\begin{claim}
			\label{claim:toobig}
			$G$ contains an anticomplete pair $(A,B)$ with $\rho(A),\rho(B)\ge (1-2w^{-2})\rho(G)$.
		\end{claim}
		\begin{subproof}
			Let $s\ge0$ be such that $2^{2^s}\ge w$. Observe that for all integers $r$ with $0\le r\le s$,
			\[\phi_{r,s}=\prod_{r<i\le s}(1-2^{-2^{i+1}})
			\ge 1-\sum_{r<i\le s}2^{-2^{i+1}}\ge 1-2^{-1-2^r}\ge1/2.\]
			
			Now, by \cref{lem:incrho}, $G$ contains either:
			\begin{itemize}
				\item an anticomplete pair $(A,B)$ with $\rho(A),\rho(B)\ge (1-2^{1-2^{s+1}})\rho(G)\ge(1-2w^{-2})\rho(G)$;
				
				\item a complete pair $(X,Y)$ with $$\rho(X)\ge\phi_{r,s}\cdot 2^{-2^{r+3}}\rho(G)
				\qquad\text{and}\qquad
				\rho(Y)\ge\phi_{r,s}\cdot\left(1-\frac52\cdot 2^{-2^r}\right)\rho(G)$$
				for some integer $r$ with $1\le r\le s$ (so $s\ge1$); or
				
				\item a complete blockade $(B_1,\ldots,B_k)$ with $k\ge2$ and $\rho(B_i)\ge\phi_s\cdot k^{-a}\rho(G)$ for all $i\in[k]$.
			\end{itemize}
			
			If the second bullet holds then
			\[\begin{aligned}
				\rho(X)&\ge \phi_{r,s}\cdot 2^{-2^{r+3}}\rho(G)\ge2^{-1-2^{r+3}}\rho(G)\ge2^{-9\cdot 2^r}\rho(G),\\
				\rho(Y)&\ge\phi_{r,s}\cdot\left(1-\frac52\cdot 2^{-2^r}\right)\rho(G)
				\ge(1-2^{-1-2^r})\left(1-\frac52\cdot 2^{-2^r}\right)\rho(G)
				\ge (1-3\cdot2^{-2^r})\rho(G)
			\end{aligned}\]
			and so the first outcome of the theorem holds with $y=2^{-2^r}\le\frac14$, contrary to our assumption.
			
			If the third bullet holds then $\rho(B_i)\ge \phi_s\cdot k^{-a}\rho(G)\ge k^{-1-a}\rho(G)\ge k^{-d}\rho(G)$ for all $i\in[k]$ (note that $k\ge 2$); and so the second outcome of the theorem holds, contrary to our assumption.
			
			Hence, the first bullet above holds, proving \cref{claim:toobig}.
		\end{subproof}
		
		Now, \cref{claim:toobig} gives an anticomplete pair $(A,B)$ in $G$ with $\rho(A),\rho(B)\ge (1-2w^{-2})\rho(G)$.
		Among all such pairs, choose one such that $G[A],G[B]$ are connected and $\abs A+\abs B$ is maximal.
		Since $G$ is connected, there is a minimal nonempty cutset $S$ separating $A,B$ in $G$.
		By the maximality of $\abs A+\abs B$, $G[A]$ and $G[B]$ are then two of the components of $G\setminus S$.
		Let $v\in S$; then since $G[A],G[B]$ are connected, $S$ is minimal, and $G$ is $P_5$-free, we see that $v$ is complete to $A$ or to $B$ in $G$, which yields $\rho(N_G(v))\ge\min(\rho(A),\rho(B))\ge(1-2w^{-2})\rho(G)$.
		Then the first outcome of the lemma holds with $X=\{v\}$, $Y=N_G(v)$, and $y=w^{-2}\le\frac14$, a contradiction.
		This proves \cref{thm:niam}.
	\end{proof}

	\section{Disjoint union}
	\label{sec:union}
	We conclude the paper with a short proof that the poly-$\chis$-bounding property, the poly-$\rho$-bounding property, and the off-diagonal \erh{} property are all closed under disjoint union.
	To do so, we apply the following \dd asymmetric R\"odl\ee{} result~\cite[Theorem 6.7]{2025thes}:
	\begin{theorem}
		\label{thm:arodl}
		For every graph $H$, there exists $C\ge1$ such that for every $\delta,\eps,\eta\in(0,\frac12)$ satisfying
		\[\delta=2^{-C\log\eps^{-1}\log\eta^{-1}}
		=\eps^{C\log\eta^{-1}},\]
		every $H$-free graph $G$ has an $\eta$-sparse or $(1-\eps)$-dense induced subgraph with at least $\delta\abs G$ vertices.
	\end{theorem}
	We will also need the following \dd trees and linear anticomplete pairs\ee{} theorem~\cite{MR4170220}, which generalises \cref{thm:p5sparse}:
	\begin{theorem}
		[Chudnovsky--Scott--Seymour--Spirkl]
		\label{thm:tree}
		For every forest $T$, there exists $\eta>0$ such that every $\eta$-sparse graph $G$ with $\abs G\ge 2$ contains an anticomplete pair $(A,B)$ with $\abs A,\abs B\ge \eta\abs G$.
	\end{theorem}
	Combining the two above results gives:
	\begin{theorem}
		\label{thm:wanti}
		For every forest $T$, there exists $d\ge2$ such that every non-complete graph $T$-free graph $G$ contains an anticomplete pair $(A,B)$ with $\abs A,\abs B\ge \omega(G)^{-d}\abs G$.
	\end{theorem}
	\begin{proof}
		Let $\eta>0$ be given by \cref{thm:tree} and let $C\ge2$ be given by \cref{thm:arodl} with $H=T$;
		we claim that $d:=3C\log\eta^{-1}$ suffices.
		To this end, let $G$ be a non-complete $T$-free graph, and let  $w:=\omega(G)$.
		If $\abs G\le w^d$ then we are done by taking $A,B$ to be the singletons of the endpoints of any nonedege of $G$;
		and so we may assume $\abs G>w^d$.
		By \cref{thm:arodl} with $\eps=w^{-2}$, $G$ has an $\eta$-sparse or $(1-w^{-2})$-dense induced subgraph $F$ with $\abs F\ge w^{-2C\log\eta^{-1}}\abs G>w^{d-2C\log\eta^{-1}}\ge w^2$ by the choice of $d$.
		Hence, if $F$ is $(1-w^{-2})$-dense then $w=\omega(G)\ge\omega(F)\ge\frac{\abs F}{1+w^{-2}\abs F}>\frac12w^2\ge w$, a contradiction.
		Thus $F$ is $\eta$-sparse;
		and so the choice of $\eta$ gives an anticomplete pair $(A,B)$ in $F$ (and so in $G$) with $\abs A,\abs B\ge \eta\abs F\ge \eta\cdot w^{-2C\log\eta^{-1}}\abs G\ge w^{-3C\log\eta^{-1}}\abs G=w^{-d}\abs G$.
		This proves \cref{thm:wanti}.
	\end{proof}
	We are now ready to show that the poly-$\rho$-bounding property is closed under disjoint union, via the following stronger result:
	\begin{theorem}
		\label{thm:rhoanti}
		For every forest $T$, there exists $d\ge2$ such that every non-complete $T$-free graph $G$ contains an anticomplete pair $(A,B)$ with $\rho(A),\rho(B)\ge \omega(G)^{-d}\rho(G)$.
	\end{theorem}
	\begin{proof}
		We claim that $d\ge2$ given by \cref{thm:wanti} suffices.
		To see this, by shrinking $G$ if necessary, we may assume $\psi(G)=\rho(G)$. Let $w:=\omega(G)$.
		By the choice of $d$, there is an anticomplete pair $(A,B)$ in $G$ with $\abs A,\abs B\ge w^{-d}\abs G$;
		and so $\rho(A)\ge \psi(A)=\abs A/\alpha(A)\ge w^{-d}\abs G/\alpha(G)=w^{-d}\psi(G)=w^{-d}\rho(G)$
		and similarly $\rho(B)\ge w^{-d}\rho(G)$.
		This proves \cref{thm:rhoanti}.
	\end{proof}
	The above argument can be straightforwardly modified to show that the off-diagonal \erh{} property is closed under disjoint union; we omit the proof.
	However proving this for the poly-$\chis$-bounding property requires a little more work, as follows.
	\begin{theorem}
		\label{thm:chisanti}
		For every forest $T$, there exists $d\ge2$ such that every non-complete $T$-free graph $G$ contains an anticomplete pair $(A,B)$ with $\chis(A),\chis(B)\ge \omega(G)^{-d}\chis(G)$.
	\end{theorem}
	\begin{proof}
		Let $T'$ be the forest obtained from $T$ by adding a leaf to each vertex of $T$;
		then no two vertices in $T'$ have a common neighbour.
		Let $d\ge2$ be given by \cref{thm:rhoanti} with $T'$ in place of $T$; we claim that $d$ suffices.
		To see this, let $G$ be a non-complete $T$-free graph, and let $w:=\omega(G)$.
		If $\chis(G)\le w^{d}$ then we are done;
		and so we may assume $\chis(G)>w^{d}$.
		
		Now, let $f\colon V(G)\to\mab N_0$ be such that $f(V(G))>0$ and $f(V(G))\ge\chis(G)\cdot f(S)$ for all stable sets $S$ of $G$ where equality holds for some such $S$.
		Let $J$ be the graph obtained from $G$ by blowing up each $v\in V(G)$ by a stable set $S_v$ with $\abs{S_v}=f(v)$.
		Since $G$ is $T$-free, it is $T'$-free; and so $J$ is $T'$-free because $T'$ has no two vertices with a common neighbour.
		Also, $J$ is non-complete, $\abs J=f(V(G))$, $\alpha(J)=f(S)$ (so $\psi(J)=\chis(G)$), and $\omega(J)\le\omega(G)\le w$.
		By the choice of $d$, there is an anticomplete pair $(A,B)$ in $J$ with $\rho(A),\rho(B)\ge w^{-d}\psi(J)>1$.
		We may assume $J[A]$ and $J[B]$ are connected;
		and so each of them has at least one edge.
		Let $A':=\{v\in V(G):S_v\cap A\ne\emptyset\}$ and $B':=\{v\in V(G):S_v\cap B\ne\emptyset\}$.
		Since each of $J[A],J[B]$ is connected and has at least one edge, $A'$ and $B'$ are disjoint and anticomplete to each other in $G$.
		Moreover $\chis(A)\ge\rho(A)\ge w^{-d}\psi(J)=w^{-d}\chis(G)$
		and $\chis(B)\ge \rho(B)\ge w^{-d}\psi(J)=w^{-d}\chis(G)$.
		This proves \cref{thm:chisanti}.
	\end{proof}
	\section*{Acknowledgement}
	We would like to thank Alex Scott and Paul Seymour for helpful discussions and encouragement.
	
	\bibliographystyle{abbrv}

\end{document}